\documentclass[12pt,reqno]{amsart}

\usepackage{amsmath}
\usepackage{amsfonts}
\usepackage{amsthm}
\usepackage{amstext}
\usepackage{amssymb}
\usepackage{txfonts}
\usepackage{enumerate}
\usepackage{graphicx,color}
\usepackage{setspace}
\usepackage{txfonts}
\usepackage{epsfig}
\usepackage{color}

\usepackage{verbatim}

\newcommand{\R}{\mathbb{R}}

\newcommand{\eps}{\varepsilon}

\renewcommand{\div}[1]{\operatorname{div}{ #1 }}
\DeclareMathOperator*{\med}{median} 

\theoremstyle{plain}
\newtheorem{theorem}{Theorem}

\newtheorem{proposition}{Proposition}
\newtheorem{conjecture}{Conjecture}

\theoremstyle{definition}
\newtheorem{definition}{Definition}

\begin{document}

\title[Medians and $1$-harmonic functions]{Median values, $1$-harmonic functions, and functions of least gradient}

\author[M. Rudd]{Matthew B. Rudd}
\email{mbrudd@sewanee.edu}
\address{Department of Mathematics \\
Sewanee: The University of the South \\
Sewanee, TN ~ 37383}

\author[H. Van Dyke]{Heather A. Van Dyke}
\email{vandykeheather@yahoo.com}
\address{Department of Mathematics \\
Washington State University \\
Pullman, WA ~  99164}

\date{\today}
\begin{abstract}
Motivated by the mean value property of harmonic functions, we introduce the local and global median value properties 
for continuous functions of two variables.  We show that the Dirichlet problem associated with the local median 
value property is either easy or impossible to solve, and we prove that continuous functions with this property are $1$-harmonic in the 
viscosity sense.   We then close with the following conjecture: a continuous function having the global median value property and prescribed boundary values 
coincides with the function of least gradient having those same boundary values.
\end{abstract}

\maketitle

\begin{section}{Introduction}

Let $\Omega \subset \R^{N}$ be a domain with smooth boundary $\partial \Omega$, 
let $B(x,r)$ denote a ball with radius $r$, center $x$, and boundary $\partial B(x,r)$, 
and let $\fint_{E}{ f \, d\mu}$ denote the average of $f$ over the set $E$ with respect to the
measure $\mu$.  It is well-known
that the continuous function $u$ is harmonic in $\Omega$ if and only if 
\begin{equation} \label{mean-vp}
u(x) = \fint_{\partial B(x,r)}{ u(s) \, ds }  \quad \textrm{whenever} \quad x \in \Omega \quad 
\textrm{and} \quad B(x,r) \Subset \Omega \, ;
\end{equation}
this is the famous mean value property of harmonic functions \cite{evans:pde98}.
As is also well-known, imposing a Dirichlet boundary condition determines a unique solution of
the functional equation (\ref{mean-vp}).
This equation has been well understood for a long time, and 
it is natural to wonder what happens if we replace the averaging operator in (\ref{mean-vp}) 
with a different statistical measure.   

What happens, for example, if we replace the average in (\ref{mean-vp}) with the
median of $u$ over spheres centered at $x$?   Requiring $u$ to
be continuous makes this question meaningful; as discussed below, medians are not
well-defined for functions that are merely measurable or integrable.

When $N=2$, which we 
assume henceforth, we can use elementary tools 
to analyze continuous solutions of the functional equation
\begin{equation} \label{loc-mvp}
u(x) = \med_{s \in \partial B(x,r)}{ \left\{ \, u(s) \, \right\} } \quad \textrm{for} \quad
x \in \Omega \quad \textrm{and} \quad 0 < r \leq R(x) \, ,
\end{equation}
where $R : \Omega \to \R$ is a continuous  function such that
\[
0 < R(x) \leq \operatorname{dist}{(x, \partial \Omega)} \, .
\]
Since $R(x)$ can be strictly smaller than the distance $\operatorname{dist}{(x, \partial \Omega)}$
from $x$ to $\partial \Omega$,  equation (\ref{loc-mvp}) defines the 
\textit{local median value property}, a property that does not seem to have been 
studied before.  We will show that equation (\ref{loc-mvp}) is closely related to 
a  nonlinear partial differential equation (PDE) that has a reputation for being 
difficult.   The degeneracy of this particular PDE precludes direct applications of 
standard techniques, making the basic problems of existence, uniqueness, and regularity
of its solutions rather thorny.  Our analysis of equation (\ref{loc-mvp}) clarifies some of these issues:  
if we impose a Dirichlet boundary condition on solutions of (\ref{loc-mvp}), we will see that 
we should generally expect either nonexistence or nonuniqueness, depending on the interaction 
of the boundary data and  the geometry of $\partial \Omega$.  

The mean value property (\ref{mean-vp}) of harmonic functions is global, in the sense that 
it holds at $x \in \Omega$ for any ball $B(x,r)$ that is strictly contained in $\Omega$.  
The global version of (\ref{loc-mvp}) is also of interest;  functions with 
the \textit{global median value property} satisfy the functional equation
\begin{equation} \label{glob-mvp}
u(x) = \med_{s \in \partial B(x,r)}{ \left\{ \, u(s) \, \right\} } \quad \textrm{whenever} \quad 
x \in \Omega \quad \textrm{and} \quad B(x,r) \Subset \Omega \, .
\end{equation}
Combining equation (\ref{glob-mvp}) with a Dirichlet boundary condition yields a problem
that is harder to solve than its local counterpart.
We conjecture that the solutions of this global problem are precisely the functions
of least gradient, for which existence and uniqueness results are only available when $\Omega$ 
is strictly convex \cite{ziemer:flg99}.  

\end{section}

\begin{section}{The local median value property}

If $f$ is a measurable function with respect to the measure $\mu$, then $m$ is a median of $f$ over the 
measurable set $E$ if and only if 
\[
\mu\{ x \in E : f(x) \geq m \} \geq \frac{ \mu( E ) }{2}  \quad\textrm{and} \quad
\mu\{ x \in E : f(x) \leq m \} \geq \frac{ \mu( E ) }{2}  \ \ .
\]
It is easy to see that such a median $m$ need not be unique.
If $f$ is integrable, then $m$ is a median of $f$ over $E$ if and only if \cite{stroock:pta93}
\[
\int_{E}{ | f(x) - m | \, d\mu(x) } = \min_{y \in \R}{ \left\{ \, \int_{E}{ | f(x) - y | \, d\mu(x)} \, \right\} } \, ,
\]
but integrability still does not guarantee a unique median.  Using this variational characterization, however,
Noah proved in \cite{noah:mcf08} that continuity  suffices to determine median values uniquely.   
By different methods, Waksman and Wasilewsky proved the same result in \cite{waksman:tll77} for continuous 
functions on planar domains.  By working exclusively with continuous functions, we are 
therefore assured that medians over spheres are well-defined and that the functional equations (\ref{loc-mvp})
and (\ref{glob-mvp}) make sense.

To understand (\ref{loc-mvp}), we first observe that any continuous 
function whose level sets are monotonically ordered straight lines has the 
local median value property.  On the other hand, if $u$ is a smooth function 
such that $Du(x_{0}) \neq 0$ at  $x_{0} \in \Omega$, then  the Implicit Function Theorem 
guarantees that, on some sufficiently small ball $B(x_{0},r)$, 
the level sets of $u$ are smooth curves whose levels vary monotonically; if $u$ also satisfies
(\ref{loc-mvp}), the level curve through $x_{0}$ must be a diameter of $B(x_{0},r)$. 
The level sets of a smooth solution of (\ref{loc-mvp}) must therefore be straight lines 
wherever its gradient does not vanish.   Our first result, Theorem \ref{straight}, eliminates this smoothness 
assumption.  Before going through its proof,  we note the following obvious and useful result, 
a weak maximum principle for solutions of (\ref{loc-mvp}).

\begin{proposition} \label{max-p}
A nonconstant continuous solution of (\ref{loc-mvp}) cannot have strict interior maxima or minima.
\end{proposition}

Figure \ref{ridge_line} illustrates a variant of this maximum principle that will be essential below.  It 
depicts a continuous function $u$ that has a ``ridge line," a curve along which $u = a$ such that $u < a$ on either
side (of course, $u$ could just as well be larger than $a$ on either side of this curve).  Our proof of Theorem \ref{straight} 
relies on the fact that such a situation cannot hold if $u$ has the local median value property everywhere.

\begin{figure}[htbp]
\includegraphics{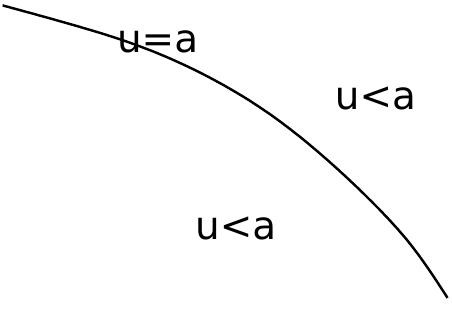} 
\caption{A continuous function cannot satisfy  (\ref{loc-mvp}) along a ridge line.}
\label{ridge_line}
\end{figure}

\begin{theorem} \label{straight}
Let $u$ be a nonconstant continuous solution of (\ref{loc-mvp}).  If $K$ is a path-connected 
component of some level set of $u$, then $\overline{ \partial K \cap \Omega } $ consists of line
segments with endpoints on $\partial \Omega$.
\end{theorem}

\begin{proof}
\quad   
Let $K$ be a path-connected component of the level set $u^{-1}(a)$,
and let $x \in K$.    The local median value property guarantees that points in $K$ 
cannot be isolated; consequently, at least one nontrivial path must pass through $x$.

\textbf{Step 1.} \quad Suppose that, for some $r >0$, $\overline{B(x,r)} \cap u^{-1}(a)$ consists of one simple path $\Gamma$
with endpoints $y, z \in \partial B(x,r)$.  Decrease $r$ (if necessary) so that (\ref{loc-mvp})
holds on $\partial B(x,r)$.   By the continuity of $u$ and the  maximum principle,
$y$ and $z$ must divide $\partial B(x,r)$ into two open arcs such
that $u > a$ on one of them and $u < a$ on the other.  The local median value property forces these two arcs to
have the same length; $y$ and $z$ are therefore antipodal points and $\Gamma$ must be 
a diameter of $B(x,r)$.

\smallskip

\begin{figure}[htbp]
\includegraphics{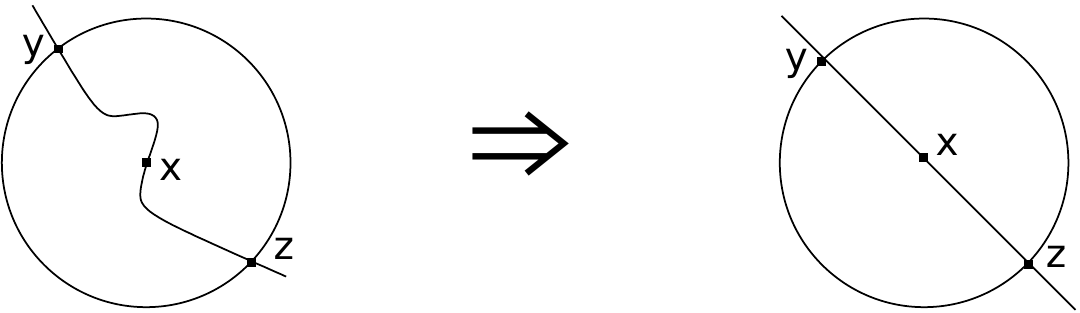}
\caption{A simple path through $x$ must be straight.}
\label{step1}
\end{figure}

\textbf{Step 2.} \quad Now suppose that, for some $r > 0$, $\overline{B(x,r)} ~ \cap ~ u^{-1}(a) \setminus \{ \, x \, \} $ consists of 
simple paths that are not intersected by any other paths.  It follows from the preceding argument that these paths must be line 
segments emanating from $x$ (Figure \ref{step2_1}).  

\begin{figure}[htbp]
\centering
\includegraphics{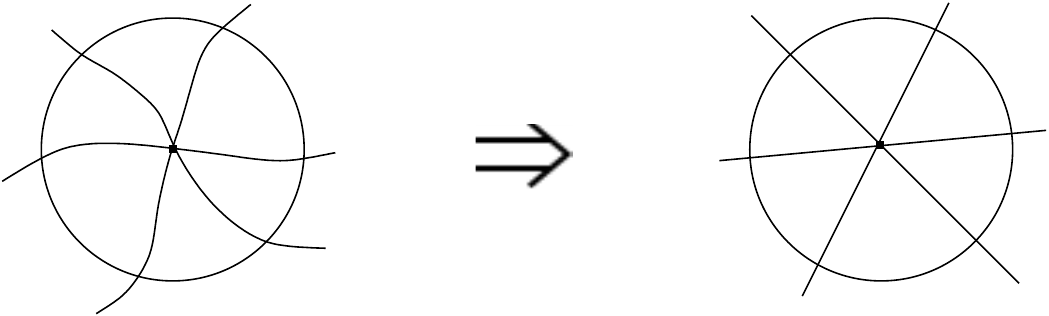}
\caption{Separable paths through $x$ must be straight.}
\label{step2_1}
\end{figure}

Continuity and  the maximum principle prohibit an odd number of line segments from meeting at $x$, since $u$ must
be alternately larger than $a$ and smaller than $a$ in the sectors determined by these segments.  Moreover,
four or more segments cannot meet at $x$ : if the angle between two adjacent segments were less than or equal to $\pi/2$, then
the local median value property would be violated (see Figure \ref{step2_2}).  Applying Step 1, we conclude that the only path through $x$ inside $B(x,r)$ is a diameter of $B(x,r)$.

\begin{figure}[htbp]
\centering
\includegraphics{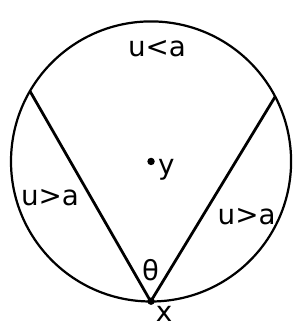}
\caption{If $\theta$ were less than $\pi/2$, then (\ref{loc-mvp}) would force $u(y) > a$, a contradiction;
if $\theta$ were exactly $\pi/2$, then we would have $u(y) = a$, another contradiction.  Consequently, $\theta > \pi/2$.}
\label{step2_2}
\end{figure}

\textbf{Step 3.} \quad  If the paths emanating from $x$ can be separated, then the previous step shows that, on some sufficiently small 
neighborhood of $x$, there can be only one straight path through $x$.  In fact, the observations in Step 2 show that, if two of the paths emanating from 
$x$ cannot be separated, then $K$ must have a nonempty interior.  Consequently, if the interior of $K$ is empty, $K$ must be a line segment.
By the local median value property, the endpoints of this segment cannot be in $\Omega$; its endpoints must be on $\partial \Omega$.
The theorem therefore holds for path-connected components with no interior.

\textbf{Step 4.} \quad Finally, suppose that the interior of $K$ is nonempty, pick a point $y \in K$, and let $z$ be a point in $K$ such that
the segment $L$ connecting $y$ and $z$ is inside $\Omega$.    
If $L$ did not belong to $K$, then the continuity of $u$ would guarantee the existence of a point $p$ on $L$ 
whose level set component had empty interior.  By the previous step, that component would be a line segment with
endpoints on the boundary of $\Omega$,  separating $y$ from $z$ and contradicting the fact that $y$ and $z$ belong to $K$.
This completes the proof.

\end{proof}

\end{section}

\begin{section}{The Dirichlet problem}

We now seek a continuous solution $u$ of (\ref{loc-mvp}) that has prescribed values along
the boundary of $\Omega$.  Theorem \ref{straight} tells  us how to go about solving the resulting Dirichlet
problem: roughly speaking, the level sets of $u$ need to be line segments or polygons.

\begin{figure}[htbp]
\centering
\includegraphics{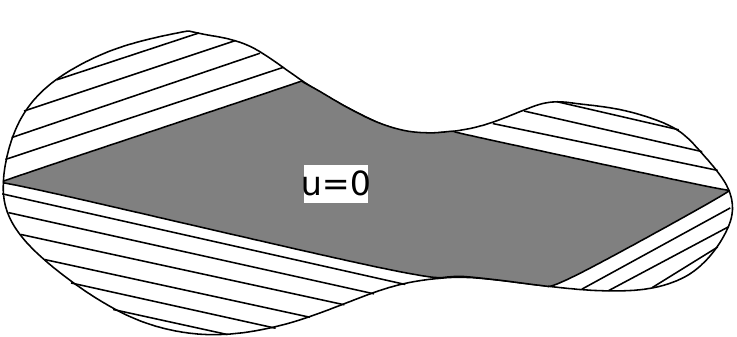}
\caption{A solution of the Dirichlet problem on a nonconvex domain $\Omega$.}
\label{nonconvex1}
\end{figure}

Figure \ref{nonconvex1} illustrates a solution of this problem for a nonconvex domain $\Omega$ 
and very particular boundary data.  This example is artificial, however; in general, we cannot expect to solve
the Dirichlet problem on a region $\Omega$ that is not strictly convex, precisely because we cannot connect arbitrary 
pairs of boundary points with line segments that remain inside $\Omega$.  See, for instance, Figure \ref{nonconvex2}.

\begin{figure}[htbp]
\centering
\includegraphics{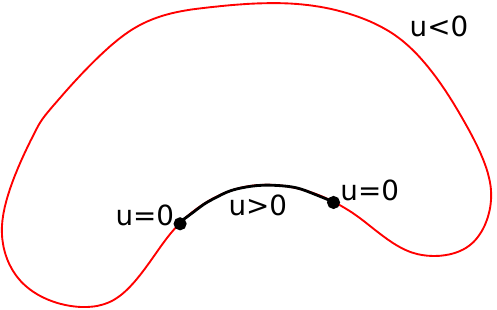}
\caption{The Dirichlet problem cannot necessarily be solved if $\Omega$ is not strictly convex.  Here, we cannot connect
positive boundary values with line segments inside $\Omega$.}
\label{nonconvex2}
\end{figure}

If $\Omega$ is strictly convex, on the other hand, solving the Dirichlet problem is much easier.
As an example, let $\Omega \subset \R^{2}$ be the open unit ball, and let $g(x,y) = |y|$ 
for $(x,y) \in \partial \Omega$.  It is easy to see that, for any 
$\alpha \in (0,1)$, the function
\begin{equation} \label{u-alpha}
u_{\alpha}(x,y) := \left\{ \begin{array}{ccl}
|y|  & \textrm{when} & |y| \geq \alpha \, , \\
\sqrt{1-x^{2}} & \textrm{when} & |y| < \alpha \quad \textrm{and} \quad 
|x| \geq \sqrt{1 - \alpha^{2}} \, , \\
\alpha & \textrm{otherwise} & 
\end{array} \right.
\end{equation}
has the local median value property and satisfies $u = g$ on $\partial \Omega$.
Figure \ref{alpha-plots} provides contour plots of solutions for three values of $\alpha$, 
illustrating the fact that different solutions of the Dirichlet problem correspond to the 
different ways of connecting the boundary values with straight lines.  

\begin{figure}[htbp]
\centering
\includegraphics{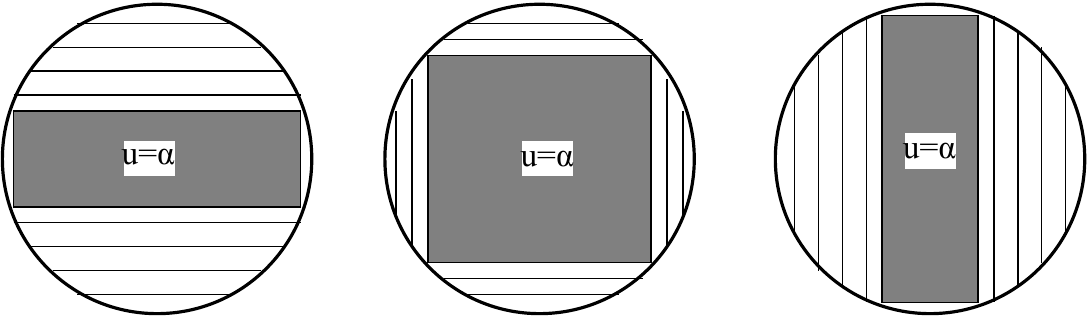}
\caption{Solutions $u_{\alpha}$ of the Dirichlet problem on the unit ball with $u_{\alpha}(x,y) = |y|$ 
on its boundary.  From left to right, $\alpha < \sqrt{2}/2$,  $\alpha = \sqrt{2}/2$, and $\alpha > \sqrt{2}/2$.}
\label{alpha-plots}
\end{figure}

An existence result for reasonable boundary data follows directly. 
Specifically, let $\Omega$ be strictly convex and bounded, and 
suppose that $g : \partial \Omega \to \R$ is a nonconstant continuous function whose level sets have finitely many endpoints.
(If $g$ is constant,  it follows from the maximum principle that the only solution of the Dirichlet problem is constant.)
By parametrizing $\partial \Omega$, we can regard $g$ as a continuous $2\pi$-periodic function, $g : [0,2\pi] \to \R$.
Define $m$ and $M$ to be the minimum and maximum of $g$, respectively, pick any $\lambda \in (m,M)$, 
and define $\Lambda := g^{-1}( \lambda )$, the set of angles $\theta$ such that $g(\theta) = \lambda$.  Since
$\Lambda$ has finitely many endpoints, $\partial \Lambda = \{ \theta_{1}, \ldots, \theta_{k} \}$ for some $k$,
with $\theta_{1} < \theta_{2} < \ldots < \theta_{k}$.

Construct a generalized polygon whose vertices on $\partial \Omega$ correspond to the angles $\theta_{1}, \ldots, \theta_{k} $.
This is a generalized polygon since some of its edges could be curved portions of $\partial \Omega$; this will happen whenever 
$g(\theta) \equiv \lambda$ on the closed interval $[\theta_{j},\theta_{j+1}]$.
As for the other edges, we can proceed as follows.   Suppose, without loss of generality, that $g(\theta) > \lambda$ for $\theta_{j} < \theta < \theta_{j+1}$.  
For sufficiently small $\eps > 0 $, the level set $g^{-1}(\lambda + \eps) \cap [ \, \theta_{j}, \theta_{j+1} \, ]$ will consist of exactly two angles,
which we connect with a line segment.  If this level set never contains more than two angles as $\eps$ increases, then we are done with
this sector of $\Omega$.  Otherwise, there will be a smallest $\eps > 0$ such that 
the boundary of $g^{-1}(\lambda + \eps) \cap [ \, \theta_{j}, \theta_{j+1} \, ]$ consists 
of angles $\phi_{1}, \ldots, \phi_{\ell}$ for $\ell \geq 3$.   Construct a generalized polygon with vertices corresponding to these
$\ell$ angles, and then repeat this procedure in each resulting sector of $\Omega$.  The compactness of $\overline\Omega$ guarantees that
this process will terminate, yielding a solution of the Dirichlet problem; its continuity
follows directly from the continuity of $g$. We have thus proven 

\begin{theorem}
If $\Omega \subset \R^{2}$ is a strictly convex, bounded open set and $g : \partial \Omega \to \R$ is a nonconstant continuous
function whose level sets have finitely many endpoints, then
there is at least one continuous function $u$ that satisfies (\ref{loc-mvp}) and equals $g$ on $\partial \Omega$.
\end{theorem}

Implementing this procedure for different values of $\lambda$ can produce different solutions; in Figure \ref{alpha-plots} 
above, the three solutions correspond to choosing $\lambda = \alpha$ for three 
different values of $\alpha$. 
Consequently, we generally expect to find multiple solutions when $\Omega$ 
is strictly convex, and an intriguing problem for future work is to quantify the number of solutions that correspond to a 
given boundary condition.   For a domain $\Omega$ that is not strictly convex, it would also be interesting 
to characterize those boundary conditions for which solutions exist. 

\end{section}

\begin{section}{$1$-harmonic functions}

For $p > 1$, the continuous function $u$ is $p$-harmonic if and only if it is a viscosity solution of
\begin{equation} \label{p-harmonic}
- \div{ \left( \, |Du|^{p-2} Du \, \right) } = 0 \, .
\end{equation}
Juutinen et al. proved \cite{juutinen:evs01} that this is equivalent to the more common definition of 
$p$-harmonic functions as weak solutions of (\ref{p-harmonic}), the Euler-Lagrange equation
corresponding to the variational integral
\begin{equation} \label{p-func}
v \longmapsto \int_{\Omega}{ |Dv|^{p} \, dx } \, \quad \textrm{for} \quad v \in W^{1,p}(\Omega) \, .
\end{equation}
When $p>1$, the variational approach to (\ref{p-harmonic}) is especially elegant, since the associated 
functional (\ref{p-func}) is strictly convex and weakly lower semicontinuous and the Sobolev space $W^{1,p}(\Omega)$
is reflexive.  

Formally setting $p=1$, one can define a $1$-harmonic
function $u$ to be a viscosity solution of 
\[
-\div{\left(\frac{Du}{ \, |Du| \, }\right)} = 0 \, ,
\]
an equation that is not nearly as well understood as (\ref{p-harmonic}).  It is more complicated than (\ref{p-harmonic})
in several regards; the functional (\ref{p-func}) is no longer strictly convex, for example, and $W^{1,1}(\Omega)$ is not
as hospitable a setting for functional analysis as $W^{1,p}(\Omega)$ (as explained further in, e.g., \cite{evans:mtf92} and \cite{giusti:msf84}).

For various reasons, it is more convenient for us to define $u$ to be 1-harmonic if and only if it is a viscosity solution of 
\begin{equation} \label{1-harm} 
-\Delta_1 u := -|Du| \div{\left(\frac{Du}{ \, |Du| \, }\right)} = 0 \, .
\end{equation}
As clarified below in Definition \ref{visc}, we can loosely interpret this equation to mean that 
either $|Du| = 0$ or ~ $\div{\left(Du / |Du| \right)} = 0$, and we will see that 
continuous solutions of (\ref{loc-mvp}) are $1$-harmonic.  We first define viscosity solutions of (\ref{1-harm}), referring 
to \cite{juutinen:evs01} for more details when $p > 1$.

\begin{definition} \label{visc}
\quad The continuous function $\overline{u} : \Omega \to \R$ is a viscosity supersolution of (\ref{1-harm}) if and only if
\begin{enumerate}
\item $ \overline{u} \nequiv \infty$ and 
\item whenever $x_0\in\Omega$ and $\varphi\in C^2(\Omega)$ satisfies
\begin{equation}\label{supsol}\left\{\begin{array}{l}
\overline{u}(x_0)=\varphi(x_0) \, ,  \\
\overline{u}(x) > \varphi(x) \quad \mbox{for} \quad x\neq x_0, \quad \textrm{and} \\
D\varphi(x_0)\neq 0 \, , \\
\end{array}\right.
\end{equation}
$\varphi$ also satisfies 
\[
- \Delta_{1}\varphi(x_{0}) \geq 0 \, .
\]
\end{enumerate}
The continuous function $\underline{u}$ is a viscosity subsolution of (\ref{1-harm}) if and only if $-\underline{u}$ is a viscosity supersolution of (\ref{1-harm}), 
and $u$ is 1-harmonic in the viscosity sense if $u$ is both a viscosity subsolution and a viscosity supersolution of (\ref{1-harm}).
\end{definition}

The work in \cite{juutinen:evs01} motivates this definition, as Juutinen et al. proved that, when $p>1$, we only need to work with test functions 
whose gradients do not vanish at the point $x_{0}$ of interest.  It then follows that our definition of $1$-harmonic and the usual one are equivalent.

\begin{theorem} \label{LMVP-1harm}
\quad A continuous function $u : \Omega \to \R $ that satisfies (\ref{loc-mvp}) is $1$-harmonic in the viscosity sense.
\end{theorem}

\begin{proof} 
Let $x_{0} \in \Omega$, and let $\varphi\in C^2(\Omega)$ be a test function that touches $u$ from below at $x_0$. 
Using (\ref{loc-mvp}), (\ref{supsol}) and the monotonicity of the median, we have  
\[
\med_{\partial s \in B(x_0,r)}{ \left\{ \, \varphi(s) \, \right\}  } < \med_{s \in \partial B(x_0,r)}{ \left\{ \, u(s) \, \right\}  } = u(x_0) = \varphi(x_0) \, ,
\]
so that
\begin{equation} \label{test-below}
\varphi(x_0) - \med_{s \in \partial B(x_0,r)}{ \left\{ \, \varphi(s) \, \right\}  } > 0 \, . 
\end{equation}
Identity (11) in \cite{hartenstine:asc11}  establishes that
\[
\varphi(x_0) - \med_{s \in \partial B(x_0,r)}{ \left\{ \, \varphi(s) \, \right\}  } = -\frac{r^2}{2} \Delta_{1} \varphi(x_{0}) + o( r^2 ) \, ,
\]
so dividing by $r^{2}/2$ and sending $r \to 0$ in (\ref{test-below})  shows that $u$ is a 
viscosity supersolution of (\ref{1-harm}).  Similarly, $u$ is a viscosity subsolution of (\ref{1-harm}), proving that 
$u$ is 1-harmonic in the viscosity sense.
\end{proof}

The key result above, identity (11) of \cite{hartenstine:asc11}, is the reason for including the 
prefactor $|Du|$ in our definition of the $1$-Laplacian $\Delta_{1}$.  Not coincidentally, our definition of
$\Delta_{1}$ coincides with the elliptic part of the operator in the widely studied mean curvature equation,
\[
u_{t} -  |Du| \div{\left(\frac{Du}{ \, |Du| \, }\right)} = 0 \, .
\]
Previous work relating this parabolic equation, medians, and median-like  
operators (\cite{cao:gce03},\cite{catte:msm95},\cite{kohn:dcb06},\cite{oberman:cmd04},\cite{ruuth:cgm00})  
motivated the present work.  

\end{section}

\begin{section}{Functions of least gradient}

In a series of papers summarized in \cite{ziemer:flg99}, Ziemer et al. analyzed  
functions of least gradient, i.e., functions $v \in BV(\Omega) \cap C(\overline{\Omega})$ 
with minimal total variation and prescribed values along $\partial \Omega$.
They proved, in particular, that a unique function of least gradient exists if $\Omega$
is strictly convex and the boundary function $g : \partial \Omega \to \R$ is continuous.   
One can assemble this
function of least gradient by constructing its level sets: they are minimal surfaces whose
boundaries are determined by the levels of $g$.  

This, of course, is exactly how we constructed solutions of the Dirichlet problem earlier, 
since one-dimensional minimal surfaces in the plane are lines.  Our method did
not produce a unique solution, however, as we only considered local conditions; the global 
constraint of minimizing total variation renders functions of least gradient unique.   
It seems reasonable,
then, that there should be a strong connection between the global median value property 
(\ref{glob-mvp}) and functions of least gradient.  To that end,  we propose
\begin{conjecture}
Suppose that  $\Omega$ is strictly convex and that $g : \partial \Omega \to \R$ is continuous,
and let $u^{*}$ be the function of least gradient on $\Omega$ that equals $g$ on $\partial \Omega$.
\begin{enumerate}
\item
There exists a unique continuous solution $u$ of (\ref{glob-mvp}) such that 
$u=g$ on $\partial \Omega$.
\item
$u = u^{*}$.
\end{enumerate}
\end{conjecture}

It is easy to see that this conjecture holds in certain cases.  For example, among the family
$\{ u_{\alpha} \}$ of solutions of (\ref{loc-mvp})  defined by equation
(\ref{u-alpha}) earlier, only $u_{1/\sqrt{2}}$ has the global median value property.
It follows easily from the coarea formula that this function's total 
variation is smaller than that of any other $u_{\alpha}$, and it is indeed the function
of least gradient on the unit circle with this boundary data.

Progress on this conjecture (and on other questions related to the results of this paper) will likely benefit from 
Juutinen's work \cite{juutinen:pha05} on $p$-harmonic functions and functions of least gradient.

\end{section}

\bigskip
\bigskip

\end{document}